\documentclass[11pt]{article}
\usepackage[latin1]{inputenc}
\usepackage{amsmath,amsthm,amssymb}
\usepackage{amsfonts}
\usepackage{amsmath,amsthm,amssymb,amscd}
\usepackage{latexsym}
\usepackage{color}
\usepackage{graphicx}
\usepackage{enumitem}
\usepackage{mathrsfs}
\usepackage{mathtools}
\usepackage{xcolor}

\makeatletter \@addtoreset{equation}{section} \makeatother

\newtheorem{theorem}{Theorem}[section]
\newtheorem{definition}{Definition}[section]

\newtheorem{lemma}{Lemma}[section]
\newtheorem{remark}{Remark}[section]

\begin{document}	
	
	\title{\sc   On subelliptic equations on  stratified Lie groups driven by  singular nonlinearity and  weak $L^1$ data}	
	\author{\sc Subhashree Sahu$^{a}$, Debajyoti Choudhuri$^{b}$, Du\v{s}an D. Repov\v{s}$^{c,d,e}$\footnote{Corresponding author: dusan.repovs@guest.arnes.si}\\
		\small{$^{a}$Department of Mathematics, National Institute of Technology Rourkela,}\\[-0.2cm]
		\small{769008, Rourkela, Odisha, India.  {\it 522ma6008@nitrkl.ac.in}}\\[-0.2cm]
		\small{$^{b}$School of Basic Sciences, Indian Institute of Technology Bhubaneswar,}\\[-0.2cm]
		\small{752050, Khordha, India.  {\it dchoudhuri@iitbbs.ac.in}}\\[-0.2cm]
		\small{$^c$Faculty of Education, University of Ljubljana, 1000, Ljubljana, Slovenia.}\\[-0.2cm]
		\small{$^d$  Faculty of Mathematics and Physics, University of Ljubljana, 1000, Ljubljana, Slovenia }\\[-0.2cm]
		\small{$^e$ Institute of Mathematics, Physics and Mechanics, 1000, Ljubljana, Slovenia.}}	 	
	
	\maketitle
	\begin{abstract}
		\noindent The article is about an elliptic problem defined on a {\it stratified Lie group}. Both sub and superlinear cases are considered whose solutions are guaranteed to exist in light of the interplay between the nonlinearities and the weak $L^1$ datum. The existence of infinitely many solutions is proved  for 
		suitable values
		 of $\lambda, p, q$ by using the Symmetric Mountain Pass Theorem. 
		\begin{flushleft}
		\end{flushleft}
	\end{abstract}	
	
	\section{Introduction}\label{s1}
	In this article, we study the existence of  solutions to the following problem:
	\begin{equation}\tag{P}\label{main}
		\left\{\begin{aligned}
			-\Delta _{p,\mathbb{G}}\,u&= f(x)u^{-\eta} + \lambda |u|^{q-2} u~\text{in}~\Omega,\\
			u&>0 ~\text{in}~\Omega,\\
			u&=0~\text{on}~\partial\Omega,
		\end{aligned}\right.
	\end{equation}
	where  $\Delta_{p,\mathbb{G}}\,u:=\operatorname{div}_{\mathbb{G}}(|\nabla_{\mathbb{G}}u|^{p-2}\nabla_{\mathbb{G}}u)$ is the sub elliptic  $p$-Laplacian and $\Omega$ is any bounded domain on a stratified Lie group  $\mathbb{G}.$ Here $0 < \eta < 1 <p <Q$, where $Q$ is the homogeneous dimension of   $\mathbb{G}.$ We consider $f$ to be a nonnegative function in some Marcinkiewicz
	 spaces
	 $\mathcal{M}^r (\Omega)$ for $r\geq 1$. We shall establish the existence of weak solutions for two different cases: $(\text{i})$~$\lambda<0$, $1<p<q$ and $(\text{ii})$~$\lambda>0$, $1<q<p$.
	
	It is worth recalling  that singular local semilinear elliptic problems such as 
	\begin{equation}\label{main_sing_prob}
		\left\{\begin{aligned}
			-\Delta \,u&= f(x)u^{-\eta}~\text{in}~\Omega,\\
			u&>0 ~\text{in}~\Omega,\\
			u&=0~\text{on}~\partial\Omega,
		\end{aligned}\right.
	\end{equation}
	arise in various physical problems, such as chemical heterogeneous catalysts
	(see e.g.,
	 Aris  \cite{aris1})
	  and
	   non-Newtonian fluids
	   (see e.g.  Emden-Shaker \cite{emden1}).
	    The term $u^{\eta}$ describes the resistance of the material and  is used in the modelling of the heat conduction in electrically conducting materials (see e.g., Fulks-Maybee \cite{fulks1} and Nachman-Callegari \cite{nachman1}). 
	
	However, our treatment of the subject will be confined within the boundaries of mathematics. No discussion of singular semilinear PDEs can be  complete without mentioning the seminal work of Lazer-McKenna \cite{lazer1}. The reader can
	 also refer to the work of 
	 Arora et al. \cite{arora1}, 
	 Biswas et al. \cite{biswas1}, 
	 Molica Bisci-Ortega \cite{bisci1}, 
	 Sbai-Hadfi \cite{sbai1}, 
	 Zuo et al. \cite{zuo}, 
	 and the references therein. 
	 The readers can consult Giacomoni et al. \cite{giacomoni1} for their approach to handling the singular term. 
	 
	 More recently,  researchers turned their attention to PDEs on {\it stratified Lie groups} (see the definition in Section \ref{s2}). One may wonder as to why this is an interesting consideration? The challenge here lies in checking if the results that hold for Euclidean domains (commutative setup) continue to hold for a noncommutative setup for a domain.
	  Ghosh et al. \cite{ghosh1} have recently reproduced a theory of fractional Sobolev spaces that opened the gateway to considering nonlocal elliptic PDEs. 
	
	The main results proved in this paper are as follows. 
	
	\begin{theorem}\label{p1theorem4.1}
		For every $0<\eta<1<p<Q$, $-\Lambda<\lambda<\infty$ with $|\Lambda|<<1,$ and every nonnegative function $f\in L^1_w (\Omega)\backslash\{0\}$,   problem \eqref{main} has at least one weak solution $u_\eta$, which is positive on $W^{1,p}_{0} (\Omega)\cap L^{\infty}(\Omega).$ Moreover, when $-\Lambda<\lambda<0,1<p<q$,   problem \eqref{main} admits a unique weak solution.
	\end{theorem}	
	
	\begin{theorem}\label{p1thm5.2}
		For every $1<q<p<Q,$  there exists $0<\Lambda<\infty$ such that for every $\lambda\in (\frac{1}{p}, \Lambda),$    problem \eqref{main} has a sequence of nonnegative weak solutions $\{u_n\} \subset X$ and $u_n \to 0 $ on $X.$ 
	\end{theorem}	

	The paper is structured as follows. 
	In Section \ref{s2}, we discuss some mathematical preliminaries and definitions.
	In Section \ref{s3}, we consider an approximation problem and discuss the existence of solutions.  
	In Section \ref{s4}, we prove Theorem \ref{p1theorem4.1}. 
	In Section \ref{s5}, we prove  Theorem \ref{p1thm5.2}.
	
	\section{Preliminaries}\label{s2}
	In this section, we recall some general facts related to  stratified Lie groups and the Marcinkiewicz space.
	For all other fundamental material used in this paper,
	 we refer the reader to the comprehensive monograph by Papageorgiou et al.  \cite{PRR}.
	We begin by referring to the readers the article by  Ghosh et al. \cite[Definition 1,2]{ghosh1} which  covers the basic definitions about homogeneous and stratified Lie group. 
	A working knowledge on these topics can be found in Garain-Ukhlov \cite{Garain}
	and
	 Hajlasz-Koskela \cite{HajlaszKoskela00}.
	
	Let $\Omega$ be an open subset of $\mathbb{G}$. The (horizontal) Sobolev space $W_{\mathbb{G}}^{1,p}(\Omega)$, $1\le p\le\infty$, consists of the locally integrable functions $u:\Omega\to\mathbb R$ having  the weak derivatives $X_if$, $1\le i\le N_1$, and
	the following
	 finite norm
	$$
	\|u\|_{W^{1,p}(\Omega)}:=\|u\|_{L^p(\Omega)}+\|\nabla_{\mathbb{G}}u\|_{L^p(\Omega)}\;.
	$$
	Here, $ \nabla_{\mathbb{G}}u= (X_1u,X_2u,\cdots ,X_{N_1}u)$ is the horizontal subgradient  of $u$. 
	
	The (horizontal) Sobolev space $W_{\mathbb{G},0}^{1,p}(\Omega)$ is the closure of $C^\infty_c(\Omega)$ in $W^{1,p}_{\mathbb{G}}(\Omega)$, equipped with the restriction $\|\ \|_{W^{1,p}_0(\Omega)}$ of $\|\ \|_{W^{1,p}(\Omega)}$.  Then $W_{\mathbb{G},0}^{1,p}(\Omega)$ is a real separable and uniformly convex Banach space. The interested reader can consult e.g.,  Folland \cite{Folland75}, { Vodop'yanov \cite{Vodop'yanov95}, Vodop'yanov-Chernikov} \cite{Vodop'yanovChernikov95}, and Xu et al. \cite{Xu90}. \\
	
	We shall refer to  
	Capogna et al. \cite[Theorem $2.3$]{Garofalo}, 
	Danielli \cite[Theorem $2.8$]{Danielli91}, 
	Danielli \cite[Theorem $2.2, 2.3$]{Danielli95},  
	Folland \cite[Theorem $5.15$]{Folland75}, 
	and
	Hajlasz-Koskela \cite[Theorem~8.1]{HajlaszKoskela00}
	 for the following embedding result.
	\begin{lemma}\label{p1lemma2.1}
		Let $\Omega$
		 be
		 bounded on $\mathbb{G}$ and $1\le p\le Q.$
		 Then $W^{1,p}_{\mathbb{G},0}(\Omega)$ is continuously embedded in $L^q(\Omega),$ for
		 every  $1\le q\le p^*:=Qp/(Q-p)$. Moreover, the embedding is compact, for every $1\le q<p^*.$ 
	\end{lemma}
	For every  $1<p<Q,$ we 
	equip
	 the Sobolev space $W^{1,p}_{\mathbb{G},0}(\Omega)$ with the norm
	$$ \|u\|:=\|u\|_{W^{1,p}_{\mathbb{G},0}(\Omega)} = \|\nabla_{\mathbb{G}}u\|_{L^p(\Omega)}. $$
	\noindent  
	Next, we  recall the definition of the Marcinkiewicz space $\mathcal{M}^p(\Omega)$ and some results about this space.
	\begin{definition}
		For every $0 < s< \infty$, the Marcinkiewicz space $\mathcal{M}^s (\Omega)$ is the set of all measurable functions f, for which  there exists $c> 0$ such that
		$$	|\{x : |f(x)| > t\}| \le \frac{c^s}{t^s}.$$ 
		The norm on this space is defined by
		$$\|f\|_{L^s_w}:= \inf\{c>0: t^s |\{x : |f(x)| > t\}| \le c^s, \text { for every } t>0\}.$$
	\end{definition}
	\noindent One should note
	that $\mathcal{M}^s (\Omega) $ is a quasinormed linear space, for every  $0<s<\infty.$ For a proof
	 of the following theorems, see Grafakos \cite{cfa}.
	\begin{theorem}\label{p1thm2.1}
		For every $0<s<\infty$  and  $f$ in $\mathcal{M}^s(\Omega)$ such that the measure of $\Omega $ is finite, we have $\|f\|_{L^s_w(\Omega)}\le \|f\|_{L^s(\Omega)}$. Hence $L^s(\Omega)\subset \mathcal{M}^s(\Omega).$ Moreover, the following inclusions hold, for every    $0<s<r$, $$ L^r (\Omega)\subseteq \mathcal{M}^r(\Omega)\subseteq  L^s (\Omega).  $$ 
	\end{theorem}
	\noindent Next, we state the H\"{o}lder inequality for quasinorms on the Marcinkiewicz space.
	\begin{theorem}\label{p1thm2.2}
		Let  $ X$ be a measurable space  and $a_j\in \mathcal{M}^{s_j}(X)$, where $0<s_j<\infty$ and $1\le j\le k.$ Let $$ \frac{1}{s} =  \frac{1}{s_1} +  \frac{1}{s_2}+ \cdots  \frac{1}{s_k}.$$ Then $$ \|a_1 a_2 \cdots a_k \|_{L^s_w} \le s^{\frac{-1}{s}} \Pi _{j=1}^k s_j ^{\frac{1}{s_j}} \Pi _{j=1}^k \|a_j\|_{L^{s_j}_w}.  $$
	\end{theorem}
	\noindent Next  is a simple algebraic lemma due to Lucio \cite[Lemma $2.1$]{Lucio} that provides the estimates
	which will be  necessary in the sequel.
	\begin{lemma}\label{p1Lucio}
		For every $1<p<\infty$, there exists a
		constant $C$ which  depends only on $p,$ 
		such that for every $\zeta$,$\rho \in \mathbb{R}^N,$ we have $$ \langle |\zeta|^{p-2}\zeta - |\rho|^{p-2}\rho, \zeta-\rho \rangle \ge C (|\zeta| +|\rho|)^{p-2} |\zeta-\rho|^2.$$
	\end{lemma}
	
	\noindent Now we shall give the definition of $\it{ weak ~solution}$ of  problem \eqref{main}.
	\begin{definition}\label{p1def 2.4}
		A function $u\in W^{1,p}_0 (\Omega)$ is said to be a weak solution of 
		problem
		\eqref{main},
		 if $u>0$ on $\Omega,$ so
		 that for every $ K \Subset \Omega, $ there exists $\delta>0,$ depending on $K,$ such that $u\ge \delta >0$ on $K$ and for every $\psi\in C_c^1(\Omega),$ we have 
		\begin{equation}\label{p1eq2.1}
			\int_{\Omega } |\nabla_{\mathbb{G}} u|^{p-2} \nabla_{\mathbb{G}}u  \nabla_{\mathbb{G}}\psi dx=\int_{\Omega} f(x)u^{-\eta} \psi dx +\lambda \int_{\Omega} |u|^{q-2}u \psi dx.
		\end{equation}
	\end{definition}
	\noindent Henceforth, by a solution we shall mean a weak solution, defined in the sense of  Definition \ref{p1def 2.4}.
	
	\section{The approximation problem}\label{s3}
	
	In this section we shall study the following approximation problem:
	\begin{align}\label{p1eq3.1}
		\begin{split}
			-\Delta _{p,\mathbb{G}}u=& \frac{f_n}{(u^{+} + \frac{1}{n})^{\eta}} + \lambda |u|^{q-2} u~\text{in}~\Omega,\\
			u=&0 ~~\text{on}~~ \partial\Omega,
		\end{split}
	\end{align}
	where $f_n(x)= \min \{f(x),n\}$, $f\in L_w ^1(\Omega)$(or $ \mathcal{M}^1(\Omega)$), and $0<\eta<1<p<Q$. We shall  subdivide it into two cases in order to prove the existence of solutions of problem \eqref{p1eq3.1}.
	
	\begin{remark}\label{p1rk3.2}
		From Theorem \ref{p1thm2.1} we can  infer that $f\in \mathcal{M}^r,$ for every $r>1,$ which automatically implies that  $ f \in L^r$ for some $r,$ and hence $f$ is a fairly regular function. Thus, a natural upgrade of   problem \eqref{main} is to consider  irregular data, say, $f \in \mathcal{M}^1$.
	\end{remark}
		
	\subsection{Case 1: $\lambda <0, 1<p<q<p^{\ast}$.}
	
	 We begin by proving the following auxiliary lemma.	
	
	\begin{lemma}\label{p1 lemma3.1}
		Let $1<p<Q$ and suppose that $a \in L^{\infty}(\Omega) \setminus \{0\}$ is nonnegative on $\Omega$. Then there exists at least one solution $u\in W^{1,p}_{0} (\Omega)\cap L^{\infty}(\Omega) $  of the following problem  
		\begin{equation}\label{p1eq3.2}
			-\Delta _{p,\mathbb{G}}u= a +\lambda |u|^{q-2}u ~~\text{in}~~ \Omega,~~u>0~~ \text{in} ~~ \Omega,~ u=0~\text{on}~\partial\Omega.
		\end{equation}
		Moreover, for every $K \Subset \Omega$, there 
		exists
		 a constant $\delta(K)$, such that $u\geq \delta(K)>0$ on $K$.
	\end{lemma}
	\begin{proof}
		We first prove the existence of solutions. To this end,
		 we define the energy functional $J:   W^{1,p}_{0} (\Omega)\cap L^{\infty}(\Omega) \rightarrow \mathbb{R} $ as follows
		  $$ J(u) = \frac{1}{p} \int_{\Omega }|\nabla_{\mathbb{G}}u|^p dx - \int _{\Omega }a(u^+) dx -\frac{\lambda}{q}  \int _{\Omega } (u^+)^qdx.$$
		Since $a\in  L^{\infty}(\Omega)$, we obtain
		 by invoking the Sobolev embedding, 
		\begin{equation*}
			J(u) \geq \frac{1}{p} \| u\|^p -c_1 |\Omega|^{\frac{p-1}{p}} \|a\|_{L^{\infty}(\Omega)} \|u\| -\frac{\lambda}{q} \int_{\Omega} |u|^q dx,
		\end{equation*}
		and since $\lambda <0$, we get the following estimate
		\begin{equation}\label{p1eq3.3}
			J(u) \geq    \frac{1}{p} \| u\|^p -c_1 |\Omega|^{\frac{p-1}{p}} \|a\|_{L^{\infty}(\Omega)} \|u\|.
		\end{equation}
		Hence, we can conclude  that $J$ is coercive.   It is a convex $C^1$ functional and it  is weakly lower semicontinuous. 
		Therefore, $J$ has a minimizer, say $u_0,$ that solves the following equation 
		$$ 	-\Delta _{p,\mathbb{G}}u= a +\lambda |u|^{q-2}u ~~\text{in}~~ \Omega.$$
		We shall now prove that this solution $u_0$ is positive. Clearly,
		 we have 
		\begin{equation}\label{p1eq3.4}
			\int_{\Omega } |\nabla_{\mathbb{G}}\, u|^{p-2} \nabla_{\mathbb{G}}\,u   \nabla_{\mathbb{G}}\,\psi dx= \int_{\Omega} a\psi dx + \lambda\int_{\Omega} |u|^{q-2}u\psi dx, 
		\end{equation}
		for every $\psi\in W_0^{1,p}(\Omega)\cap L^{\infty}(\Omega)$. In particular, on testing \eqref{p1eq3.4} with  $\psi=u^- $, we obtain
		\begin{align}\label{pos1}
			0\ge  -\|u^-\|^p = \int_{\Omega} au^- dx - \lambda\int_{\Omega} |u^-|^q dx>0.
		\end{align}
		For a sufficiently small $\lambda<0$, the right hand side of  equation \eqref{pos1} is positive. Therefore $u^-=0$ a.e. on $\Omega$, and hence $u\geq 0$ a.e. on $\Omega$. The positivity  of $u$ follows by the Strong Minimum Principle for nonnegative super solutions. 
		One can refer to the consequence of   {Vodop'yanov} \cite[Theorem 5]{Vodop'yanov95}.
		 This completes the proof of Lemma \ref{p1 lemma3.1}.
	\end{proof}
	\begin{remark}\label{p1rk3.3}
		Note that the compact embedding from Lemma \ref{p1lemma2.1} also holds for our solution space.
	\end{remark}
	\begin{lemma}\label{p1 lemma3.2}
		For every  $n\in \mathbb{N},$   problem \eqref{p1eq3.1} has a  positive solution $u_n \in  W^{1,p}_{0} (\Omega)\cap L^{\infty}(\Omega)$, and for every  $K \Subset \Omega$, there 
		exists
		 a constant $\delta(K)$ such that $u_n\geq \delta(K)>0$ on $K$. Moreover, $\|u_n\| \leq c_2$, for some $c_2>0,$ independent of n. 
	\end{lemma}
	\begin{proof}
		By  Lemma \ref{p1 lemma3.1}, for a fixed $n\in\mathbb{N},$ and for any $g\in L^p (\Omega), $ there exists
		 a $u_n$
		  satisfying 
		\begin{equation}\label{p1eq3.5}
			\begin{aligned}
				-\Delta _{p,\mathbb{G}}u&= \frac{f_n}{(g^{+} + \frac{1}{n})^{\eta}} + \lambda |u|^{q-2} u~\text{in}~\Omega,\\
				u&=0 ~~on ~~\partial \Omega.
			\end{aligned}
		\end{equation} 
		We define $F: L^p (\Omega) \rightarrow  L^p (\Omega)  $ so that $F(g)=I(u)=u,$ where $u$ is a
		 solution of problem \eqref{p1eq3.5} and $I$ is the inclusion map from $  W^{1,p}_{0} (\Omega)\cap L^{\infty}(\Omega)  $ to $ L^p (\Omega).$ \\	
		 
		 	\noindent
		\textit{\bf Claim:}	$F$ {\it is continuous.}		\\
		Indeed, consider any $g\in L^p(\Omega)$. Then there exists a sequence $\{g_m\}$  in  $ L^p (\Omega)$  such that $g_m \to g $ in $ L^p (\Omega) .$ In order to show 
		that
		$F$ is continuous, we have to prove 
		that $u_m:=F(g_m) \to u=:F(g)$ in $ L^p (\Omega). $ By recalling the properties of the map $F$, we have for every  $\psi \in  W^{1,p}_{0} (\Omega)\cap L^{\infty}(\Omega), $
		\begin{equation}\label{p1eq3.6}
			\int_{\Omega } |\nabla_{\mathbb{G}}u_m|^{p-2} \nabla_{\mathbb{G}}u_m   \cdot\nabla_{\mathbb{G}}\psi dx= \int_{\Omega}  \frac{f_n}{(g_m^{+} + \frac{1}{n})^{\eta}} \psi dx + \lambda\int_{\Omega} |u_m|^{q-2}u_m\psi dx, 
		\end{equation}
		and 
		\begin{equation}\label{p1eq3.7}
			\int_{\Omega } |\nabla_{\mathbb{G}} u|^{p-2} \nabla_{\mathbb{G}}u  \nabla_{\mathbb{G}}\psi dx= \int_{\Omega}  \frac{f_n}{(g^{+} + \frac{1}{n})^{\eta}} \psi dx + \lambda\int_{\Omega} |u|^{q-2}u\psi dx .
		\end{equation}
		We define $$ g_{m,n}: =\left( \left(g_m ^+ + \frac{1}{n}\right)^{-\eta} -\left(g^+ + \frac{1}{n}\right)^{-\eta} \right).$$
		Next, we choose $\psi = (u_m -u)$ in \eqref{p1eq3.6} and \eqref{p1eq3.7}, then subtract the resulting equations and use the Sobolev embedding, to arrive at the following estimate
		\begin{equation}\label{p1eq3.8}
			\begin{aligned}
				&\int_{\Omega } ( |\nabla_{\mathbb{G}} u_m|^{p-2} \nabla_{\mathbb{G}}u_m -  |\nabla_{\mathbb{G}}u|^{p-2} \nabla_{\mathbb{G}}u ) ( \nabla_{\mathbb{G}}(u_m -u))dx \\
				& = \int_{\Omega}  f_n ((g_m^+ + \frac{1}{n})^{-\eta} - (g^+ +\frac{1}{n})^{-\eta}) + \lambda \int_{\Omega} (|u_m|^{q-2}u_m - |u|^{q-2}u) (u_m -u)dx\\
				&\leq n \int_{\Omega} |g_{m,n}| |u_m-u| dx + \lambda \int_{\Omega} (|u_m|^{q-2}u_m - |u|^{q-2}u) (u_m -u)dx\\
				&\leq n \|g_{m,n}\|_{L^{(p^\ast)^\prime}(\Omega)} \|u_m -u\| _{L^{p^\ast}(\Omega)}+ \lambda \int_{\Omega} (|u_m|^{q-2}u_m - |u|^{q-2}u) (u_m -u)dx\\
				&\leq c_3~ n \|g_{m,n}\|_{L^{(p^\ast)^\prime}(\Omega)} \|u_m -u\| + \lambda \int_{\Omega} (|u_m|^{q-2}u_m - |u|^{q-2}u) (u_m -u)dx,\\
			\end{aligned}
		\end{equation}
		where $c_3 >0$ is the Sobolev constant. By using  Lemma \ref{p1Lucio} in \eqref{p1eq3.8}, we can conclude that
		\begin{equation}\label{p1eq3.9}
			\|u_m -u\| \leq c_4 ~n \|g_{m,n}\|_{L^{(p^\ast)^\prime}(\Omega)} ^{\frac{1}{t-1}},
		\end{equation}
		where $t=p$ if $p\geq2,$ and $t=2$ if $p<2.$ We note that $|g_{m,n}| \leq 2 n^{\eta +1}, $ and  up to a subsequence, 
		 $g_{m,n}\to 0,$ as $m \to \infty.$ Thus, using the Lebesgue Dominated Convergence Theorem in \eqref{p1eq3.9}, we can conclude  that $u_m\to u$ in $ L^p (\Omega). $ As a result, $F$ is indeed continuous.\\
		 
		 \noindent
		\textit{\bf Claim:} 	$F$ {\it is  compact.}\\		
		Indeed, on taking $u$ as a test function in \eqref{p1eq3.5} we have,
		$$ \|u\|^p \leq \int_{\Omega} n^{\eta+1} u dx+\lambda \int_{\Omega} |u|^q dx.$$
		Invoking the compact embedding and  the fact that $\lambda <0$, we obtain
		$$ \|u\|^p \leq   C_5 n^{\eta+1} |\Omega|^{\frac{p-1}{p} }\|u\| .$$ Therefore, 
		\begin{equation}\label{p1eq3.10}
			\|u\| \leq C_6,
		\end{equation}
		where $C_6 >0$ is a constant, independent of choice of $g.$ In order to show 
		that $F$ is compact, let $\{g_m\}$ be a bounded sequence in $  L^p (\Omega).$ Then by \eqref{p1eq3.10}, we have $\|F(g_m)\| \leq C_7.$  Therefore, by the compact embedding of $ W^{1,p}_{0} (\Omega)\cap L^{\infty}(\Omega) $ in $ L^p (\Omega), $ there is a subsequence which strongly converges in $ L^p (\Omega) $. Hence, $F$ is indeed compact.		
		
		We now apply the Schauder Fixed Point Theorem, which guarantees the existence of a fixed point $u_n$ which solves problem \eqref{p1eq3.2}. Positivity of the solution follows from Lemma \ref{p1 lemma3.1}.\\
		
		\noindent
		\textit{\bf Claim}: $\{u_n\}$ {\it is uniformly bounded on} $ W^{1,p}_{0} (\Omega)\cap L^{\infty}(\Omega).$ \\		
		Indeed, by choosing $\psi= u_n $ as a test function in \eqref{p1eq3.5}, we obtain
		\begin{equation}\label{p1eq3.11}
			\|u_n\|^p\leq \int_{\Omega} f u_n^{1-\eta}dx + \lambda \int_{\Omega} |u_n|^q dx \leq  \int_{\Omega} f u_n^{1-\eta}dx .
		\end{equation}
		Let  the domain be split into two regions in such a way that $\Omega = K \cup (\Omega\setminus K),$ where $K$ is compact subset of $\Omega$. We have the following estimate for $\int_{\Omega} f u_n^{1-\eta}dx $ near the boundary of $ \Omega \backslash K, $ which follows by the definition of 
		the
		$L^1_w$- norm:
		\begin{equation}\label{p1eq3.12}
			f(x)u^{1-\eta}(x) | \{x:|f(x)|\geq f(x)u^{1-\eta}(x) \}| < \|f\|_{1,w},
		\end{equation}
		where $A_f^u:=| \{x:|f(x)|\geq f(x)u^{1-\eta} \}| $ is the measure of  the set    $$\{x:|f(x)|\geq f(x)u^{1-\eta} \} .$$ On integrating \eqref{p1eq3.12} over $ \Omega \setminus K, $ we obtain 
		\begin{equation*}
			|A_f^u|\int_{\Omega\backslash K} f u^{1-\eta}dx \leq \int_{\Omega\backslash K}  \|f\|_{1,w}dx,
		\end{equation*}
		which further implies
		\begin{equation}\label{p1eq3.13}
			\int_{\Omega\backslash K} f u^{1-\eta}dx \leq C_9 \|f\|_{1,w},
		\end{equation}
		where the constant $ C_9>0$ depends on the compact set $K$ and  the measure of $\Omega.$  Using \eqref{p1eq3.13} in \eqref{p1eq3.11}, we can conclude  that
		\begin{equation}\label{p1eq3.14}
			\begin{aligned}
				\|u_n\|^p &\leq\int_{\Omega} f u_n^{1-\eta}dx \\
				&\leq \int_K f u_n^{1-\eta}dx +\int_{\Omega\backslash K} f u_n^{1-\eta}dx\\
				&\leq \int_K f u_n^{1-\eta}dx+ C_9 \|f\|_{1,w}.
			\end{aligned}
		\end{equation}
		Since $ \{u_n\}\subset  L^{\infty}(\Omega),$ it is bounded
		on a fixed compact set $K\subset \Omega.$
		Furthermore, by Lemma \ref{p1 lemma3.1}, we have $u_n\geq \delta(K)>0.$ Using the positivity of $u_n$ on compact subset $K$ in \eqref{p1eq3.14}, we have
		\begin{equation}\label{p1eq3.15}
			\|u_n\|^p \leq C_{10} \int_K f dx+ C_9 \|f\|_{1,w},
		\end{equation}
		so by applying H\"{o}lder's inequality in \eqref{p1eq3.15}, we get
		\begin{equation}\label{p1eq3.16}
			\|u_n\|^p \leq C_{10} |K| \|f\|_{1,w} +C_9 \|f\|_{1,w},
		\end{equation}  
		therefore, we can conclude  that 
		$$\|u_n\| \leq C_{11},$$
		thereby  proving our claim that $\{u_n\}$ is uniformly bounded on $  W^{1,p}_{0} (\Omega).$ Finally, we have that  $\{u_n\}$ is uniformly bounded on $  W^{1,p}_{0} (\Omega)\cap L^{\infty}(\Omega).$ 
		This completes the proof of Lemma \ref{p1 lemma3.2}.
	\end{proof}
	
	\subsection{Case 2: $\lambda >0, 1<q<p<p^{\ast}.$}
	We shall first prove the following auxiliary lemma.
	\begin{lemma}\label{p1 lemma3.3}
		Let $1<q<p<Q$ and suppose that $b(x) \in L^{\infty}(\Omega) \backslash \{0\}$ is nonnegative on $\Omega$. Then there exists at least one solution $u\in W^{1,p}_{0} (\Omega)\cap L^{\infty}(\Omega) $  of the following  problem  
		\begin{equation}\label{p1eq3.17}
			-\Delta _{p,\mathbb{G}}\,u= b +\lambda |u|^{q-2}u ~~\text{in}~~ \Omega,~~u>0~~ \text{in} ~~ \Omega ~ u=0~\text{on}~\partial\Omega.
		\end{equation}
		Moreover, for every $\omega \subset \Omega$, there
		 exists 
		 a constant $\gamma(\omega)$ such that $u\geq \gamma(\omega)>0$ on $K$.
	\end{lemma}
	\begin{proof}
		We first  define the energy functional $I$ in the following way, to establish the existence of solutions of
		 \eqref{p1eq3.17}. Let $I:   W^{1,p}_{0} (\Omega)\cap L^{\infty}(\Omega) \rightarrow \mathbb{R} $ be 
		 defined
		  as $$ I(u) = \frac{1}{p} \int_{\Omega }|\nabla_{\mathbb{G}}\,u|^p dx - \int _{\Omega }bu dx -\frac{\lambda}{q}  \int _{\Omega } |u^q|dx.$$
		Since $b\in  L^{\infty}(\Omega)$,  using the Sobolev embedding, we have 
		\begin{equation*}
			I(u) \geq \frac{1}{p} \| u\|^p -C_{12} |\Omega|^{\frac{p-1}{p}} \|b\|_{L^{\infty}(\Omega)} \|u\| -\frac{\lambda}{q} C_{13} \|u\|^q ,
		\end{equation*}
		where $ C_{12}$ and  $C_{13}$ are the Sobolev constants. Since $\lambda >0$, we have the following estimate
		\begin{equation}\label{p1eq3.18}
			I(u) \geq   \|u\|^q (\frac{1}{p}  \| u\|^{p-q} -C_{12} |\Omega|^{\frac{p-1}{p}} \|b\|_{L^{\infty}(\Omega)} \|u\|^{1-q}-\frac{\lambda}{q} C_{13}),
		\end{equation}
		and since $1<q<p$, we can conclude  that $I$ is coercive. 		
		In a similar manner as in  Lemma \ref{p1 lemma3.1}, we can conclude that $I$ has a minimizer $u$ which solves the equation 
		$$ 	-\Delta _{p,\mathbb{G}}\,u= b+\lambda |u|^{q-2}u ~~\text{in}~~ \Omega, u|_{\partial\Omega} =0.$$
		
		We now prove that $u>0$ a.e. in $\Omega.$
		for every  $\psi \in    W^{1,p}_{0} (\Omega)\cap L^{\infty}(\Omega).$
		We have 
		\begin{equation}\label{p1eq3.19}
			\int_{\Omega } |\nabla_{\mathbb{G}}\, u|^{p-2} \nabla_{\mathbb{G}}\,u   \nabla_{\mathbb{G}}\,\psi dx= \int_{\Omega} a\psi dx + \lambda\int_{\Omega} |u|^{q-2}u\psi dx 
		\end{equation}
		Taking $\psi=u^- $ as the test function in \eqref{p1eq3.4}, we get
		$$0 \ge -\|u^-\|^p = \int_{\Omega} bu^- dx + \lambda\int_{\Omega} |u^-|^{q-2}u^- dx,$$
		and since $\lambda>0$, the right hand side of the above equation is positive, which is a contradiction. Therefore $u^-=0$, hence $u\geq 0$ a.e. in $\Omega.$ The positivity  of $u$  
		follows
		 by the Strong Minimum Principle for nonnegative super solutions. 
		 This completes the proof of Lemma \ref{p1 lemma3.3}.
	\end{proof}
	\begin{lemma}\label{p1 lemma3.4}
		For every $n\in \mathbb{N},$ problem \eqref{p1eq3.1} has a  positive solution $u_n \in  W^{1,p}_{0} (\Omega)\cap L^{\infty}(\Omega)$, and for every  $\omega \Subset \Omega$, there 
		exists
		 a constant $\gamma(\omega)$ such that $u_n\geq \gamma(\omega)>0$ on $\omega$. Moreover, $\|u_n\| \leq C_{14}$, for some $ C_{14}>0,$ which is independent of n. 
	\end{lemma}
	\begin{proof}
		By  Lemma \ref{p1 lemma3.3}, for every $n\in\mathbb{N}$ and  $h\in L^p (\Omega),$ there exists
		 $u_n$ such that 
		\begin{equation}\label{p1eq3.20}
			-\Delta _{p,\mathbb{G}}\,u= \frac{f_n}{(h^{+} + \frac{1}{n})^{\eta}} + \lambda |u|^{q-2} u~\text{in}~\Omega, u|_{\partial\Omega} =0 .
		\end{equation} 
		Define $F: L^p (\Omega) \rightarrow  L^p (\Omega)  $ so that $F(h)=I(u)=u,$  $u$ is a
		 solution of \eqref{p1eq3.20}, where $I$ is the inclusion map from $  W^{1,p}_{0} (\Omega)\cap L^{\infty}(\Omega)  $ to $ L^p (\Omega).$ \\
		 
		 \noindent
		\textit{\bf Claim:} $F$ {\it is continuous.}\\
		The proof that $F$ is continuous is very similar to that of  Lemma \ref{p1 lemma3.2}. \\
		 
		 \noindent
		\textit{\bf Claim:} $F$ {\it is  compact.}\\
		Indeed, taking $u$ as a test function in \eqref{p1eq3.20}, we have
		\begin{equation}\label{p1eqex}
			\|u\|^p \leq \int_{\Omega} n^{\eta+1} u dx+\lambda \int_{\Omega} |u|^q dx.
		\end{equation}
		Then by  using the compact embedding on the right side of \eqref{p1eqex}, we obtain
		\begin{equation}\label{p1eq3.21}
			\|u\|^p \leq   C_{15} n^{\eta+1} |\Omega|^{\frac{p-1}{p} }\|u\| +\lambda C_{16}\|u\|^q .
		\end{equation}
		Suppose  that$u$ were unbounded. Then on dividing $\|u\|^q $ in \eqref{p1eq3.21}, we would have
		$$	\|u\|^{p-q} \leq   C_{15} n^{\eta+1} |\Omega|^{\frac{p-1}{p} }\|u\|^{1-q} +\lambda C_{16}, $$ 
		and since $q<p$, we arrive at a contradiction.
		Therefore, 
		\begin{equation}\label{p1eq3.22}
			\|u\| \leq C_{17},
		\end{equation}
		where $C_{17} >0$ is a constant independent of the choice of $h.$ Working along the similar lines as in the proof of Lemma \ref{p1 lemma3.2}, we can conclude  that $F$ is a compact operator.
		 Hence for every  $n\in\mathbb{N},$   problem \eqref{p1eq3.20} admits at least one solution $u_n$.\\
		 
		 \noindent
		\textit{\bf Claim:} $\{u_n\}$ {\it is uniformly bounded on $ W^{1,p}_{0} (\Omega)\cap L^{\infty}(\Omega).$}\\ 
		Indeed, by choosing $\psi =u_n$ as a test function in \eqref{p1eq3.20}, we get
		\begin{equation}\label{p1eq3.23}
			\|u_n\|^p\leq \int_{\Omega} f u_n^{1-\eta}dx + \lambda \int_{\Omega} |u_n|^q dx
		\end{equation}
		and by using the Sobolev embedding in \eqref{p1eq3.23}, we have
		$$	\|u_n\|^p\leq \int_{\Omega} f u_n^{1-\eta}dx + C_{18}\lambda \|u_n\|^q, $$ where $C_{18}>0$ is the Sobolev constant. Now, by using  the idea of splitting  of the domain as  in  the proof of Lemma \ref{p1 lemma3.2}, we arrive at a similar situation as in \eqref{p1eq3.6}. Thus, we have the following estimate
		\begin{equation}\label{p1eq3.24}
			\|u_n\|^p \leq C_{19} \|f\|_{1,w}+  C_{18}\lambda \|u_n\|^q.
		\end{equation}
		Now suppose  that  $	\|u_n\|\to \infty, $ as $n\to \infty.$ Since $q<p$, on dividing \eqref{p1eq3.24} by $ \|u_n\|^q$, we arrive at a contradiction. Therefore, we can conclude  that  $\{u_n\}$ is uniformly bounded on $ W^{1,p}_{0}(\Omega)$, which further implies that $\{u_n\}$ is bounded on $ W^{1,p}_{0} (\Omega)\cap L^{\infty}(\Omega).$ 
		This completes the proof of Lemma \ref{p1 lemma3.4}.
	\end{proof} 
	\begin{remark}\label{p1rk3.1}
		One can observe from  Lemmas \ref{p1 lemma3.2} and \ref{p1 lemma3.4} that $\{u_n\}$ is uniformly bounded on  $ W^{1,p}_{0} (\Omega)\cap L^{\infty}(\Omega).$ So the compact embedding of   $ W^{1,p}_{0} (\Omega)\cap L^{\infty}(\Omega)$ in $L^{p}(\Omega) $ guarantees that up to a subsequence, $u_n \to u_\eta,$ as $n\to \infty , $ where $u_\eta$ is the pointwise limit of $u_n$ in $L^{p}(\Omega).$
	\end{remark}
	\section{Proof of Theorem \ref{p1theorem4.1}}\label{s4}
	In this section we shall
	prove the strong convergence of the sequence of solutions of the approximation problems   \eqref{p1eq3.5} and  \eqref{p1eq3.20}, where the limit of these solutions gives the solution of our main problem  \eqref{main} for both of  these cases. 
	\begin{lemma}\label{p1 lemma4.1}
		Let $ \{u_n\}$ be a sequence of solutions of the approximation problem \eqref{p1eq3.5}, given by  Lemma \ref{p1 lemma3.2}. By Remark \ref{p1rk3.1}, if $ u_\infty$ is the pointwise limit of $\{u_n\},$ then up to a subsequence,
		\begin{equation}\label{p1eq4.1}
			u_n \to u_\infty~~~ strongly ~in ~ W^{1,p}_{0} (\Omega)\cap L^{\infty}(\Omega).
		\end{equation}
	\end{lemma}
	\begin{proof}
		Since  $ \{u_n\}$  is the solution of \eqref{p1eq3.5},  we can conclude  from  Remark \ref{p1rk3.1} that
		$ u_n \rightharpoonup u_\infty$ in $ W^{1,p}_{0} (\Omega)\cap L^{\infty}(\Omega) $ and $ u_n \to u_\infty$ in $L^s(\Omega)~ \text{for}~ 1\leq s<p^\ast.$ On testing the weak formulation of problem \eqref{p1eq3.5}
		with $\psi=u_n,$ we get
		
		$$ \|u_n\|^p =\int_{\Omega}  \frac{f_n}{(u_n + \frac{1}{n})^{\eta}} u_n dx+\lambda\int_{\Omega} |u_n|^q dx, $$ 
		and letting $n\to \infty$, we have 
		\begin{equation}\label{p1eq4.3}
			\lim _{n\to\infty}  \|u_n\|^p = \int_{\Omega} f u_\infty ^{1-\eta} dx +\lambda \int_{\Omega} u_\infty ^q dx,
		\end{equation}
		so since $ u_n \to u_\infty$ in $L^s(\Omega) $, $u_\infty$ satisfies the weak formulation. Thus, we have the following 
		\begin{equation}\label{p1eq4.4}
			\int_{\Omega } |\nabla_{\mathbb{G}}\, u_\infty|^{p-2} \nabla_{\mathbb{G}}\,u_\infty  \cdot \nabla_{\mathbb{G}}\,\psi dx = \int_{\Omega} fu_\infty^ {-\eta} \psi dx + \lambda\int_{\Omega} |u_\infty|^{q-2}u_\infty\psi dx.
		\end{equation}
		Choosing $\psi= u_\eta $ as a test function in \eqref{p1eq4.4}, we obtain
		\begin{equation}\label{p1eq4.5}
			\|u_\infty\|^p =\int_{\Omega} f u_\infty ^{1-\eta} dx +\lambda \int_{\Omega} u_\infty ^q dx,
		\end{equation}
		thus, by virtue of  \eqref{p1eq4.3} and \eqref{p1eq4.5} along with the $\it{ uniform~ convexity}$ of  $W^{1,p}_{0} (\Omega)\cap L^{\infty}(\Omega),$ we can conclude  that 
		$$ 	\lim _{n\to\infty}  \|u_n\|^p= \|u_\infty\|^p.$$ Therefore, \eqref{p1eq4.1} follows. 
		This completes the proof of Lemma \ref{p1 lemma4.1}.
	\end{proof}
	\begin{remark} With a similar argument as in the proof of  Lemma \ref {p1 lemma4.1}, one can  prove the strong convergence of the solutions of $\{u_n\}$ of \eqref{p1eq3.20}.
	\end{remark}
	
	Next, we shall prove the existence of solutions of problem  \eqref{main} for both the sublinear and superlinear cases and an additional restriction on $\lambda.$ 
	
	\begin{proof}[Proof of Theorem \ref{p1theorem4.1}]
		For every  $n\in \mathbb{N},$ Lemmas \ref{p1 lemma3.2} and  \ref{p1 lemma3.4} guarantee the existence of $ u_n \in W^{1,p}_{0} (\Omega)\cap L^{\infty}(\Omega) $ which obeys
		\begin{equation}\label{p1eq4.6}
			\int_{\Omega } |\nabla_{\mathbb{G}}\, u_n|^{p-2} \nabla_{\mathbb{G}}\,u_n   \nabla_{\mathbb{G}}\,\psi dx= \int_{\Omega}  \frac{f_n}{(u_n + \frac{1}{n})^{\eta}} \psi dx + \lambda\int_{\Omega} |u_n|^{q-2} u_n \psi dx, \hbox{ for every } \psi \in C^1_c (\Omega).
		\end{equation}
		Using the strong convergence derived in Lemma \ref{p1 lemma4.1}, we have up to a subsequence, $\nabla_{p,\mathbb{G}} u_n \to \nabla_{p,\mathbb{G}} u_\infty $ as $n \to \infty$ pointwise a.e. in $\Omega.$ Then by passing to the limit, we have 
		\begin{equation}\label{p1eq4.7}
			\lim _{n \to \infty} 	\int_{\Omega } |\nabla_{\mathbb{G}}\, u_n|^{p-2} \nabla_{\mathbb{G}}\,u_n  \cdot \nabla_{\mathbb{G}}\,\psi dx= 	\int_{\Omega } |\nabla_{\mathbb{G}}\, u_\infty|^{p-2} \nabla_{\mathbb{G}}\,u_\infty\cdot   \nabla_{\mathbb{G}}\,\psi dx.
		\end{equation}
		Suppose that supp $\psi =K$, where $K\subset \Omega$ is compact. Then by Lemma \ref{p1 lemma3.2} (or \ref{p1 lemma3.4}), there exists a constant $\delta(K)$ such that $u_n\geq \delta(K) >0$ in $K.$ Thus, $u_\infty \ge  \delta(K) >0 $ on $ K$, and we also have
		$$ \frac{f_n}{u_n^{\eta}} \psi \leq \frac{f}{\delta(K)^\eta} \|\psi\|_{L^\infty (\Omega)}.$$ Now as mentioned in Remark \ref{p1rk3.1}, we can use the pointwise convergence of $u_n\to u_\infty$ a.e. in $\Omega$ and the Lebesgue Dominated Convergence Theorem to conclude
		\begin{equation}\label{p1eq4.8}
			\lim _{n \to \infty}\left( \int_{\Omega}  \frac{f_n}{(u_n + \frac{1}{n})^{\eta}} \psi dx + \lambda\int_{\Omega} |u_n|^{q-2} u_n \psi dx\right)= \int_{\Omega}  \frac{f}{(u_\infty )^{\eta}} \psi dx + \lambda\int_{\Omega} |u_\infty|^{q-2} u_\infty \psi dx
		\end{equation}
		and invoke \eqref{p1eq4.7} and \eqref{p1eq4.8}  in \eqref{p1eq4.6} to conclude that $u_\infty $ is a weak solution of
		problem
		 \eqref{main}, for both  considered cases, i.e. $\lambda<0,1<p<q$ and $ \lambda>0, 1<q<p.$ \\
		We shall prove the uniqueness of solutions for the case when $\lambda<0, 1<p<q.$
		Suppose the obtained solution is not unique, i.e., let $u_1, u_2$ be two solutions of \eqref{main}. Then choosing $\psi =(u_1- u_2)^+$ as a test function in \eqref{p1eq2.1}, we get
		\begin{equation}\label{p1eq4.9}
			\int_{\Omega} |\nabla_{\mathbb{G}} u_1|^{p-2} \nabla_{\mathbb{G}} u_1 \nabla_{\mathbb{G}}(u_1 -u_2)^+ dx = \int_{\Omega} f u_1^{-\eta} (u_1 -u_2)^+ dx + \lambda \int_{\Omega} |u_1|^{q-2}u_1 (u_1 -u_2)^+ dx,
		\end{equation}
		\begin{equation}\label{p1eq4.10}
			\int_{\Omega} |\nabla_{\mathbb{G}} u_2|^{p-2} \nabla_{\mathbb{G}} u_2 \nabla_{\mathbb{G}}(u_1 -u_2)^+ dx = \int_{\Omega} f u_2^{-\eta} (u_1 -u_2)^+ dx + \lambda \int_{\Omega} |u_2|^{q-2}u_2 (u_1 -u_2)^+ dx.
		\end{equation}
		On subtracting  \eqref{p1eq4.10} from\eqref{p1eq4.9}, we obtain
		\begin{equation}\label{p1eq4.11}
			\begin{aligned}
				&	\int_{\Omega} (|\nabla_{\mathbb{G}} u_1|^{p-2} \nabla_{\mathbb{G}} u_1  - |\nabla_{\mathbb{G}} u_2|^{p-2} \nabla_{\mathbb{G}} u_2 )\cdot (\nabla_{\mathbb{G}}(u_1 -u_2)^+) dx \\
				& = \int_{\Omega}  f (u_1^{-\eta}- u_2^{-\eta}) (u_1 -u_2)^+ dx +\lambda \int_{\Omega} (|u_1|^{q-2}u_1-|u_2|^{q-2}u_2 )(u_1 -u_2)^+ dx \le 0.
			\end{aligned}
		\end{equation}
		Thus, by Lemma \ref{p1Lucio}, we have $u_1\le u_2$  a.e. in $\Omega.$  Furthermore, on subtracting \eqref{p1eq4.9} from \eqref{p1eq4.10} with $ \psi=( u_2 -u_1)^+$ as a test function, we obtain by a similar argument that $u_2\leq u_1$ a.e. in $\Omega.$ Therefore, the uniqueness follows.
		 This completes the proof of Theorem \ref{p1theorem4.1}.	\end{proof}
	\begin{remark}
		Consider the following problem
		\begin{equation}\tag{Q}\label{main1}
			\left\{\begin{aligned}
				-\Delta _{p,\mathbb{G}}\,u&= f(x)|u|^{-\eta-1}u + \lambda |u|^{q-2} u~\text{in}~\Omega,\\
				u&=0~\text{on}~\partial\Omega.
			\end{aligned}\right.
		\end{equation}
		A careful observation reveals that the solution of
		problem 
		 \eqref{main} is also a solution of problem 
		\eqref{main1}. This  implies that $-u$ is also a solution of problem  \eqref{main1}. Therefore,   problem \eqref{main1} has at least two solutions. One possible direction to investigate at this stage is under what conditions does   problem \eqref{main1} have infinitely many solutions - this will lead us to the next section.  
		
	\end{remark}
	\section{Proof of  Theorem \ref{p1thm5.2}}\label{s5}
	It will be shown in this section that there
	 exist
	  infinitely many solutions of  problem \eqref{main1}, for every $\lambda>0,1<q<p,$ using the variational methods. 
	Here, the idea is to use the Symmetric Mountain Pass Theorem in some restricted domain of $\Omega$ to guarantee the existence of distinct infinitely many weak solutions.
		\noindent  In order to use the symmetric mountain pass geometry, we shall refer to Kajikiya \cite[Definition 1.1]{Kajikiya} for the definition of genus.
		Further, we recall the set,
		$$ \Gamma_k = \{S_k\subset X:S_k  ~\text{is closed and symmetric, } 0 \notin S_k,\text{ and } \gamma(S_k)\geq k \},$$ where $X$ is a Banach space and $S\subset X.$ The set $ \Gamma_k$ is the collection of closed and symmetric set around $0,$ with genus at least $k.$ 
	
	\begin{remark}
		According to Rabinowitz \cite{Rabinowitz}, roughly speaking,  genus is a tool for measuring the size of symmetric sets.
	\end{remark}
	\begin{definition}\label{p1def5.2}
		Let $X$ be any Banach space, $J\in C^1(X,\mathbb{R}),$ and $c\in \mathbb{R}.$ The function $J$  is said to satisfy the  $(PS)_{c}$ condition if every  sequence $(u_n)\subset X$ such that  $J(u_n)\rightarrow c$ and $J'(u_n)\rightarrow 0,$ as $n\rightarrow\infty,$ has a convergent subsequence.
	\end{definition}
	\noindent As a consequence, if there exists a $(PS)_c$ subsequence, still denoted by $(u_n)$, which converges strongly on $ W^{1,p}_{0} (\Omega)\cap L^{\infty}(\Omega) $, then $J$ is said to satisfy the $PS$ condition.\\
	
	\noindent Here, the following version of the Symmetric Mountain Pass Theorem by Kajikiya \cite{Kajikiya} will be considered.
	\begin{theorem}\label{p1thm5.1}
		Let $X$ be an infinite-dimensional Banach space and $\hat{J}$  a $ C^1$ functional on $X$ that satisfies the following conditions\\
		(i)  $\hat{J}$ is even, bounded below, $\hat{J}(0)= 0,$ and  $\hat{J}$ satisfies the $(PS)_c$ condition.\\
		(ii) For every  $k\in\mathbb{N},$ there exists  $S_k\in \Gamma_k$ such that $ \underset{u\in S_k}{\sup} \hat{J}(u)<0.$\\
		Then for every $k\in \mathbb{N}, c_k = \underset{S\in \Gamma_k}{\inf }~ \underset{{u\in S}}{\sup}~ \hat{J}(u)<0 $ is a critical value of $\hat{J}.$
	\end{theorem}
	For more details on various versions of the Mountain Pass Theorem, we refer to Youssef \cite{Jabri}.  
	
	To apply Theorem \ref{p1thm5.1}, we shall follow the technique from Ghosh-Choudhuri \cite{Ghosh}, for our problem \eqref{main} for the case $1<q<p$.
	On modifying \eqref{main} as follows
	\begin{equation}\label{p1eq5.1}
		\left\{\begin{aligned}
			-\Delta _{p,\mathbb{G}}\,u&= f(x) \frac{sign(u)}{|u|^\eta} + \lambda |u|^{q-2}
			u~\text{in}~\Omega\\
			u&=0~\text{on}~\partial\Omega,
		\end{aligned}\right.
	\end{equation}
	a careful observation reveals that the weak solution of  problem \eqref{p1eq5.1} is also a  weak solution of
	 problem \eqref{main1}. We use  the cutoff technique to guarantee the existence of infinitely many solutions of  \eqref{p1eq5.1} under suitable assumptions on $u$ given in
	  Ghosh et al.
	  \cite{Ghosh}, which are as follows:\\ 
	(a) There
	 exists
	 $\delta >0$ such that $\hbox{ for every } x\in\Omega$, $|u|\le \delta.$ \\
	 (b) There
	 exist
	   $r>0$ and  $\beta \in (1 - \eta, 2)$ such that,
	for every $x$ in $ \Omega$ and  $|u|<r$, we have $|u|^q\le \beta |u|^q.$
	
	By choosing $i$ so that $0<i<\frac{1}{2} \min\{\delta,r\},$ the cutoff function is defined as $ \phi : \mathbb{R}\to \mathbb{R}^+$ so that $0\leq \phi(t)\leq 1$ and 
	$$  \phi(t)=	\left\{\begin{aligned}
		&1,~~~~ \text{if}~ |t|\leq i\\
		&\phi~ \text{is decreasing, if}~~ i\leq t\leq 2i\\
		&0, ~~~~\text{if}~ |t|\geq 2i.
	\end{aligned}\right.$$
	Indeed, at first, we have  to prove  the existence of solutions of the following cutoff problem
	\begin{equation}\label{p1eq5.2}
		-\Delta _{p,\mathbb{G}}u= f(x) \frac{\hbox{sign}(u)}{|u|^\eta} + \lambda |u|^{q-2}u \phi(u),
	\end{equation}
	with $\|u\|_{L^\infty}\leq i.$ By Theorem \ref{p1theorem4.1}, the  solution of 
	problem 
	\eqref{main}  is positive. We only need to show 
	 that $u$ is bounded by $ i$. \\
	The energy functional $\hat{J}: W^{1,p}_{0} (\Omega)\cap L^{\infty}(\Omega)\to \mathbb{R} $ associated   to   problem \eqref{p1eq5.2} is defined as
	\begin{equation}\label{p1eq5.3}
		\hat{J}(u)= \frac{1}{p} \int_{\Omega } |\nabla_{\mathbb{G}}u|^p dx- \frac{1}{1-\eta} \int_{\Omega } f|u|^{1-\eta} dx -\lambda\int_{\Omega } G(u)dx,
	\end{equation}
	where $G(t)=\int_0^t |s|^{q-2}s \phi(s) ds.$
	We now prove the  following auxiliary lemma.
	\begin{lemma}\label{p1 lemma5.1}
		For every  $0<\eta<1,$  let us define
		$$\Lambda= \inf\{\lambda>0: \text{problem \eqref{main} has no weak solutions}\}.$$
		Then $ 0\leq \Lambda<\infty.$
	\end{lemma}
	\begin{proof}
		Invoking Garain-Ukhlov \cite{Garain1}, let $\lambda_0$ be the eigenvalue of the operator $ (-\Delta_{p,\mathbb{G}})$ in $\Omega$ and let $\xi$ be the associated eigenfunction. Then
		\begin{equation}\label{p1eq5.4}
			\begin{aligned}
				(-\Delta_{p,\mathbb{G}}) \xi&= \lambda_0 \xi \hbox{ in } \Omega,\\
				\xi&>0, \\
				\xi&=0 \hbox{ on } \partial \Omega.
			\end{aligned}
		\end{equation}
		Using $\xi\ge0$ as a test function in the weak formulation of \eqref{p1eq5.4}, we have
		\begin{equation}\label{p1eq5.5}
			\lambda_0 \int_{\Omega} u \xi dx=	\int_{\Omega} ( \Delta_{p,\mathbb{G}} \xi) u dx=	\int_{\Omega}\left ( \frac{f(x)}{|u|^\eta} + \lambda |u|^{q-2}u \right)\xi dx .
		\end{equation}
		Since solution exists for every $\lambda$,  we can choose a sufficiently large  $\lambda =\Lambda_1 >0$ for which $|t|^{q-2}t+\Lambda_1t^{-\eta} >2\lambda_0t,$ for every $ t>0$. This is a contradiction to \eqref{p1eq5.5}. Hence  we can conclude that $\Lambda<\infty.$
		 This completes the proof of Lemma \ref{p1 lemma5.1}.
	\end{proof}
	We now prove the following two lemmas necessary to apply the Symmetric Mountain Pass Theorem.
	\begin{lemma}\label{p1 lemma5.2}
		The functional $\hat{J}$ satisfies $(PS)_c$ condition and
		 is
		 bounded below in $W^{1,p}_{0} (\Omega)\cap L^{\infty}(\Omega).$
	\end{lemma} 
	\begin{proof}
		By using the definition of $\phi$ in \eqref{p1eq5.3}, we have 
		\begin{equation}\label{p1eq5.6}
			\hat{J}(u)\geq \frac{1}{p} \|u\|^p - \frac{1}{1-\eta} \int_{\Omega } f|u|^{1-\eta} dx -\lambda C_{20}.
		\end{equation}
		Let $K\Subset \Omega,$ then by using the same idea of splitting domain like in the proof of Lemma \ref{p1 lemma3.2}, we have from \eqref{p1eq3.13} 
		\begin{equation}\label{p1eq5.7}
			\begin{aligned}
				\int_{\Omega} f u^{1-\eta}dx 
				&= \int_K f u^{1-\eta}dx +\int_{\Omega\setminus K} f u^{1-\eta}dx\\
				&\leq \int_K f u^{1-\eta}dx+ C_{21} \|f\|_{1,w}  
			\end{aligned}
		\end{equation}
		and since $ u\in  L^{\infty}(\Omega)$,  using the H\"{o}lder inequality, we get
		\begin{equation}\label{p1eq5.8}
			\int_K f u^{1-\eta}\leq C_{22}  \|f\|_{1,w},
		\end{equation} 
		hence by invoking \eqref{p1eq5.7} and \eqref{p1eq5.8}  in \eqref{p1eq5.6}, we can conclude  that
		$$  	\hat{J}(u)\geq \frac{1}{p} \|u\|^p - \frac{C_{23}}{1-\eta} \|f\|_{1,w} -\lambda C_{20}.  $$ Thus, $\hat{J}$ is coercive and bounded below in $W^{1,p}_{0} (\Omega)\cap L^{\infty}(\Omega).$ 
		
		Suppose now that $\{u_n\}$ 
		is
		 a Palais-Smale sequence for  $\hat{J}.$
		Then  $\{u_n\} $ is bounded due to the coerciveness of $\hat{J}.$ 
		 Therefore, up to a subsequence, $u_n \rightharpoonup u$ in $W^{1,p}_{0} (\Omega)\cap L^{\infty}(\Omega)$. So we have, for every
		  $\psi \in W^{1,p}_{0} (\Omega)\cap L^{\infty}(\Omega) ,$ 
		$$  \int_{\Omega } |\nabla_{\mathbb{G}} u_n|^{p-2} \nabla_{\mathbb{G}}u_n   \nabla_{\mathbb{G}}\psi  dx \to \int_{\Omega } |\nabla_{\mathbb{G}} u|^{p-2} \nabla_{\mathbb{G}}u   \nabla_{\mathbb{G}}\psi  dx, $$
	this using the compact embedding, we can conclude that $ u_n \to u $ in $L^p(\Omega).$ Finally, by Lemma \ref{p1 lemma4.1} we can conclude  that $\|u_n\| \to \|u\|.$
	 This completes the proof of Lemma \ref{p1 lemma5.2}.
	\end{proof}
	\begin{lemma}\label{p1 lemma5.3}
		For every $k\in \mathbb{N},$  there exists a symmetric, closed subset $S_k \subset  W^{1,p}_{0} (\Omega)\cap L^{\infty}(\Omega) $ with $0\notin S_k,$ such that  $\gamma(S_k)\geq k,$ and  for  every $ \frac{1}{p}<\lambda <\Lambda$, we have $\sup_{u\in S_k} \hat{J}(u)<0.$
	\end{lemma}
	\begin{proof}
		Let us define $X:=W^{1,p}_{0} (\Omega)\cap L^{\infty}(\Omega), $ and let $X_m$ denote any finite-dimensional subspace of $X$ such that dim$(X_m)=m.$ By the equivalence of norms on a finite-dimensional space, we have $\|u\| \leq M \|u\|_{L^p(\Omega)},~ \hbox{for every } u\in X_m.$ To prove that there exists constant $R>0$ such that 
		\begin{equation}\label{p1eq5.9}
			\frac{1}{p}\int_{\Omega} |u|^p dx \geq \int_{\{|u|\geq i\}} |u|^p dx, \hbox{ for every } u\in X_m , \|u\| \leq R,
		\end{equation}
		we use the method of contradiction. Let $ \{ u_n\}$ be a sequence in $X_m \setminus \{0\}$ such that $u_n \to 0$ in $X$ and $$ \frac{1}{p}\int_{\Omega} |u_n|^p dx < \int_{\{|u_n|\geq i\}} |u_n|^p dx.$$ Let us define $w_n = \frac{u_n}{\|u_n\|_{L^p(\Omega) }}.$ Then we have 
		\begin{equation}\label{p1eq5.10}
			\frac{1}{p} < \int_{\{|u_n|\geq i\}} |w_n|^p dx.
		\end{equation}
		Since $X_m$ is  finite-dimensional, we have up to a subsequence , that
		$w_n \to w$ in $X,$ which further implies  that
		 $w_n \to w$ also in $L^p (\Omega),$ as $n\to\infty.$ Since $ u_n\to 0,$ we have
		$$ | \{x\in \Omega : |u_n|>i\}| \to 0, ~~\text {as}~~ n\to \infty$$ 
		which contradicts \eqref{p1eq5.10}.\\
		Now, since $q<p,$ we have $\underset{t\to 0}\lim \frac{|t|^{q-2}t}{t^p}= \infty . $ Therefore, we can choose $0<i\leq 1$ such that 
		$$  G(u) = |u|^q\geq 2M^p u^p,$$
		so for every $u\in X_m \setminus\{0\}$ such that $\|u\| \leq R,$ using \eqref{p1eq5.9}, we obtain 
		\begin{equation*}
			\begin{aligned}
				\hat{J}(u) & \leq \frac{1}{p} \int_{\Omega } |\nabla_{\mathbb{G}}u|^p dx- \frac{1}{1-\eta} \int_{\Omega } f|u|^{1-\eta} dx -\lambda\int_{\{|u|\leq i \} } G(u)dx \\
				& \leq   \frac{1}{p} \|u\|^p -\frac{1}{1-\eta} \int_{\Omega } f|u|^{1-\eta} dx -2M^p \lambda\int_{\{|u|\leq i \} } |u|^p dx \\
				&= \frac{1}{p} \|u\|^p -\frac{1}{1-\eta} \int_{\Omega } f|u|^{1-\eta} dx -2M^p \lambda \left(\int_{\Omega } |u|^p dx- \int_{\{|u|>i \} } |u|^p dx\right)\\
				& \leq \frac{1}{p} \|u\|^p -\frac{1}{1-\eta} \int_{\Omega } f|u|^{1-\eta} dx -M^p \lambda  \int_{\Omega } |u|^p dx\\
				&\leq M^p (\frac{1}{p}-\lambda)  \|u\|^p_{L^{p} (\Omega)} -\frac{1}{1-\eta} \int_{\Omega } f|u|^{1-\eta} dx.
			\end{aligned}
		\end{equation*}
		For every  $\frac{1}{p}<\lambda<\Lambda $, we have $ \hat{J}(u)<0,$ so se now choose $0<r\leq R$ and $S_k=\{u\in X_k:\|u\|=r\}.$ This yields $\Gamma_k \neq\phi.$ Since $S_k$ is symmetric and closed with $\gamma(S_k)\geq k,$ it follows that $\underset{u\in S_k}\sup \hat{J}(u)<0.$
		 This completes the proof of Lemma \ref{p1 lemma5.3}.
	\end{proof}
	
	Finally, we can prove our second main result.
	
	\begin{proof}[Proof of Theorem \ref{p1thm5.2}]
		By the definition of $\hat{J}$ and the way we defined our cutoff function $\phi,$ we have that $ \hat{J}$ is even and $ \hat{J}(0)=0.$ By Theorem \ref{p1thm5.1} and Lemmas \ref{p1 lemma5.1} and  \ref{p1 lemma5.2}, we can conclude  that $ \hat{J}$
		has sequence of critical points $ \{u_n\}$ such that $ \hat{J}(u_n)<0$ and  $\hat{J}(u_n)\to 0^-. $ The positivity and boundedness of $\{u_n\}$  follow from Lemma \ref{p1 lemma3.4}.\\
		Moreover, by the definition of $\hat{J},$  
		\begin{equation*}
			\begin{aligned}
				\frac{1}{\beta} \langle \hat{J}^\prime(u_n),u_n\rangle -  \hat{J}(u_n) =	&\frac{1}{\beta}  \left[  \|u_n\|^p - \int_{\Omega }\left( f\frac{ sign(u_n)u_n}{|u_n|^\eta} +\lambda  |u_n|^q \phi(u_n)\right)dx\right] \\
				&\quad -\left[  \frac{1}{p} \|u_n\|^p- \frac{1}{1-\eta} \int_{\Omega } \left(f|u|^{1-\eta} -\lambda G(u)\right)dx, \right]\\
				&= \left(	\frac{1}{\beta}-\frac{1}{p}  \right) \|u_n\|^p - \left(	\frac{1}{\beta}- \frac{1}{1-\eta}\right) \int_{\Omega } f|u|^{1-\eta} dx\\
				& \quad + 	\frac{\lambda}{\beta} \int_{\Omega } (\beta G(u_n)-|u_n|^{q-2}u_n\phi(u_n)) dx\\
				&\geq \left(	\frac{1}{\beta}-\frac{1}{p}  \right) \|u_n\|^p + \left(\frac{1}{1-\eta}-	\frac{1}{\beta}\right)\int_{\Omega } f|u|^{1-\eta} dx\\
				&\geq \left(	\frac{1}{\beta}-\frac{1}{p}  \right) \|u_n\|^p,
			\end{aligned}
		\end{equation*}
		so by a simple observation,
		 $$ 	\frac{1}{\beta} \langle \hat{J}^\prime(u_n),u_n\rangle -  \hat{J}(u_n) = o_n(1),$$ which implies $$\left(	\frac{1}{\beta}-\frac{1}{p}  \right) \|u_n\|^p \leq  o_n(1),$$
		 thus, $ u_n \to 0$ in $X.$ Now, using the fact that $u_n\in L^\infty (\Omega),$ we  apply the Moser iteration technique to conclude that $\|u_n\|_{L^\infty (\Omega)} \leq i,$ as $n\to \infty.$ Therefore,   problem \eqref{p1eq5.2} has infinitely many solutions, which further guarantees that problem  \eqref{main} has infinitely many solutions for every $\frac{1}{p}<\lambda<\Lambda.$
		  This completes the proof of Theorem \ref{p1thm5.2}.
	\end{proof}
	\begin{remark}
		We have proved the existence and uniqueness of solutions for the superlinear case in Theorem \ref{p1theorem4.1} for sufficiently small $\lambda<0$, hence infinitely many solutions of   problem \eqref{main} will cease to
		 exists.
		  However, for a larger magnitude of $-\lambda,$ even the existence of solutions cannot be established, and therefore we have not discussed this case in this paper.
	\end{remark}
	
	\begin{remark} 
		As an illustrative example, we consider the
		 singular
		 semilinear Dirichlet boundary value problem studied by Perera-Silva \cite{perera1}:
		\begin{equation}\label{ex1}
			\left\{\begin{aligned}
			-\Delta_p u=&a(x)u^{-\eta}+\lambda f(x,u)~\text{in}~\Omega,\\
			u>&0~\text{in}~\Omega,\\
			u=&0 ~\text{on}~\partial\Omega.
			\end{aligned}\right.
	\end{equation}
	They considered the case when $\eta>0$, whereas we consider the case when $\eta\in (0,1)$. However, our analysis implies that   problem \eqref{ex1} has infinitely many solutions, which is an improvement over the result in \cite{perera1} for the case when $\eta \in(0,1)$.
	\end{remark}	
	
	\subsection*{Acknowledgements}
	Sahu acknowledges the funding received from the National Board for Higher Mathematics (NBHM) and Department of Atomic Energy (DAE) of India [02011/47/2021/NBHM(R.P.)/R\&D II/2615]. Repov\v{s} acknowledges the funding received from the Slovenian Research and Innovation Agency [P1-0292, J1-4031, J1-4001, N1-0278]. 
	We thank the referee for comments and suggestions.

\end{document}